\documentclass[journal]{IEEEtran}
\usepackage{amsmath}
\usepackage{amssymb}
\usepackage{algorithm}
\usepackage{tcolorbox}
\usepackage{tikz}
\usepackage[noend]{algpseudocode}
\usepackage{graphicx}
\usepackage{pgfplots}
\usepackage{mathtools}
\usepackage[noadjust]{cite}
\usepackage{subcaption}
\usepackage{pgfplotstable}
\usepackage{todonotes}
\usepackage{url}
\usepackage{hyperref}

\pgfplotsset{compat=1.18}

\algrenewcommand\textproc{\textsc}

\DeclareMathOperator*{\argmin}{arg\,min}

\DeclareMathOperator{\col}{col}

\title{Reference Condensation for\\Model Predictive Control with Preview}

\author{Daniel Arnstr{\"o}m%
\thanks{E-mail: daniel.arnstrom@gmail.com.}}

\begin{document}

\definecolor{set19c1}{HTML}{E41A1C}
\definecolor{set19c2}{HTML}{377EB8}
\definecolor{set19c3}{HTML}{4DAF4A}
\definecolor{set19c4}{HTML}{984EA3}
\definecolor{set19c5}{HTML}{FF7F00}

\maketitle

\newtheorem{proposition}{Proposition}
\newtheorem{lemma}{Lemma}
\newtheorem{corollary}{Corollary}
\newtheorem{remark}{Remark}
\newtheorem{theorem}{Theorem}
\newtheorem{definition}{Definition}
\newtheorem{assumption}{Assumption}
\newtheorem{example}{Example}
\newtheorem{problem}{Problem}

\pgfplotstableread{data/step_nopreview.dat}{\stepnopreview}
\pgfplotstableread{data/step_preview.dat}{\steppreview}
\pgfplotstableread{data/step_refcond.dat}{\steprefcond}
\pgfplotstableread{data/step_virtual.dat}{\stepvirtual}
\pgfplotstableread{data/step_avgref.dat}{\stepavgref}
\pgfplotstableread{data/condensation_matrix.dat}{\condensationmatrix}
\pgfplotstableread{data/sin_nopreview.dat}{\sinnopreview}
\pgfplotstableread{data/sin_preview.dat}{\sinpreview}
\pgfplotstableread{data/sin_refcond.dat}{\sinrefcond}
\pgfplotstableread{data/sin_avgref.dat}{\sinavgref}

\begin{abstract}
In model predictive control (MPC), preview information can greatly improve tracking.
Including preview information does, however, increase the parameter dimension linearly with the preview horizon, which increases online cost and, more importantly, the complexity of explicit MPC. We introduce \emph{reference condensation}, a method that compresses a future reference trajectory into a single setpoint through a linear map. For the unconstrained tracking problem, the map follows from a least-squares projection. For receding-horizon MPC, we also study a weighted variant that prioritizes the first applied control. Numerical experiments on a double integrator and a higher-order aircraft example show that the weighted condensation closely matches full preview while keeping the parameter dimension independent of the preview horizon.
\end{abstract}


\section{Introduction}

In preview control~\cite{birla2015optimal,tomizuka1975optimal}, future setpoints are used to improve tracking. This fits naturally in model predictive control~(MPC)~\cite{rawlings2017model,borrelli2017predictive}, where a finite-horizon optimal control problem is solved at each time step and the prediction horizon already provides a window of future information. Fig.~\ref{fig:preview} highlights the benefit: with preview information, the controller anticipates the change in the reference and, hence, tracks it better.

The main limitation of preview in MPC is that it increases the number of parameters on which the optimal control action~$u^*$ depends. Without preview, the control depends on the state $x \in \mathbb{R}^{n_x}$ and one setpoint $r \in \mathbb{R}^{n_r}$, so $(x,r) \mapsto u^*$ and the parameter dimension is $n_x + n_r$. With preview over a horizon~$N$, the control depends on the full reference trajectory: $(x,r_1,\dots,r_N) \mapsto u^*$, so the dimension becomes $n_x + Nn_r$. For typical horizons ($N = 20$--$100$), this increase is large.

A higher parameter dimension affects both online and offline MPC. Online, larger parameter vectors must be communicated to, and processed at, the controller in each sampling instant. In \emph{explicit} MPC~\cite{bemporad2002explicit,alessio2009survey}, where the control law is pre-computed offline and stored as a lookup table, the effect is more severe: complexity typically grows exponentially with the number of parameters. As a result, preview is rarely used in explicit MPC, even though MPC is predictive by design.

\emph{Related work.}
Several MPC methods modify references for other reasons. Reference and command governors~\cite{garone2017reference} filter references to ensure constraint satisfaction and stability. Tracking MPC for piecewise constant references~\cite{limon2008mpc} and artificial-reference methods~\cite{krupa2024model} introduce virtual setpoints to guarantee feasibility, stability, and offset-free tracking. Unlike these approaches, we condense a reference trajectory into one setpoint to \emph{preserve preview benefits} while reducing the parameter dimension. Finite-horizon preview itself dates back to~\cite{tomizuka1975optimal}. Our contribution is a principled way to retain those benefits without carrying the full preview trajectory as a parameter.

\emph{Contribution.}
We introduce \emph{reference condensation}, which compresses a sequence of future references $r_1, \ldots, r_N$ into one setpoint through a linear map. For the unconstrained tracking problem we obtain the optimal unweighted map $\bar{r} = S \mathbf{r}$, while for receding-horizon MPC we also consider a weighted map $\bar{r}_W = S_W \mathbf{r}$ that emphasizes the first control action. Our main contributions are:
\begin{enumerate}
    \item A closed-form expression for the optimal unweighted condensation matrix $S$ (Theorem~\ref{th:closedform-refcond}), obtained by minimizing the deviation between the unconstrained preview and setpoint-tracking control sequences.
    \item Properties of the condensation maps: the condensed reference is an affine combination of the future references (Lemma~\ref{lem:affine}), the approximation error is bounded by the reference variation (Proposition~\ref{prop:error-bound}), and the induced closed-loop state mismatch is bounded (Proposition~\ref{prop:closedloop-bound}).
    \item Numerical experiments showing that reference condensation used in MPC achieves tracking performance nearly identical to full preview while reducing the parameter dimension from $n_x + Nn_r$ to $n_x + n_r$.
\end{enumerate}
From a practical perspective, reference condensation makes using preview information in explicit MPC tractable.

\emph{Notation.}
We define the weighted norm ${\|x\|_M^2} \triangleq x^\top M x$ for a matrix $M \succeq 0$. The identity matrix of size $n$ is denoted~$I_n$, and the matrix $\mathbf{I}_{n \times N}$ denotes $N$ copies of $I_n$ stacked vertically. When the dimensions are clear from context we write~$\mathbf{I}$. A vector of ones of length~$n$ is denoted $\mathbf{1}_n$. Bold font denotes stacked vectors and matrices: $\mathbf{u} \triangleq (u_0, \ldots, u_{N-1})$ and $\mathbf{r} \triangleq (r_1, \ldots, r_N)$. The Moore--Penrose pseudoinverse is denoted~$(\cdot)^\dagger$ and the largest singular value of a matrix~$M$ is denoted~$\bar{\sigma}(M)$.

The rest of the paper is organized as follows. Section~\ref{sec:problem} formulates the problem. Section~\ref{sec:refcond} derives the condensation and its properties. Section~\ref{sec:mpc} discusses constrained and explicit MPC. Section~\ref{sec:experiments} presents numerical experiments.%

\section{Problem Formulation}
\label{sec:problem}
Consider discrete-time linear systems of the form
\begin{equation}
    \label{eq:dynamics}
    x_{k+1} = A x_k + B u_k,
\end{equation}
where $x_k \in \mathbb{R}^{n_x}$ and $u_k \in \mathbb{R}^{n_u}$ denote the state and control input at time step~$k$, and $A$, $B$ are matrices of appropriate dimensions. At each time step,  our objective is that a linear combination $C x_k$ of the state should track a reference~$r_k \in \mathbb{R}^{n_r}$, where $C \in \mathbb{R}^{n_r \times n_x}$, while limiting control effort. This is formalized with the quadratic objective
\begin{equation}
    \label{eq:cost}
    J(\mathbf{u}; x_0, \mathbf{r}) = \sum_{k=1}^{N} \|C x_k - r_k\|_{Q}^2 + \sum_{k=0}^{N-1}\|u_k\|_R^2,
\end{equation}
where the weights $Q \succeq 0$ and $R \succ 0$, the tracking matrix~$C$, and the horizon $N \in \mathbb{Z}_{>0}$ are assumed to have been selected appropriately.

Given a state~$x_0$ and reference trajectory $\{r_k\}_{k=1}^N$, the optimal control sequence $\mathbf{u}^*$ is given by the finite-horizon linear-quadratic tracking problem~\cite[Ch.~4]{anderson2007optimal}:
\begin{equation}
    \label{eq:tracking-lqc}
    \begin{aligned}
        \mathbf{u}^*\!\left(x_0, \mathbf{r}\right) \triangleq  &\argmin_{\mathbf{u}} \;J(\mathbf{u}; x_0, \mathbf{r})\\
                                                               &\text{ s.t.} \; x_{k+1} = A x_k + B u_k, \; k = 0, \ldots, N{-}1,\\
        & x_0 \text{ given.}
    \end{aligned}
\end{equation}
Our goal is to find a \emph{constant} setpoint~$\bar{r}$ such that replacing the full reference trajectory with~$\bar{r}$ gives the best matching control sequence. We formalize this as follows.

\begin{tcolorbox}
\begin{problem}[Reference condensation]
    \label{prob:ref-cond}
    Given a reference trajectory $\mathbf{r} = (r_1, \ldots, r_N)$, find
    \begin{equation}
        \label{eq:optimal-condensation}
        \bar{r} = \argmin_{\bar{r}} \left\|\mathbf{u}^*\!\left(x_0,\mathbf{r}\right)-\mathbf{u}^*\!\left(x_0, \mathbf{I}\bar{r}\right)\right\|_2^2,
    \end{equation}
    where $\mathbf{I} \bar{r} = (\bar{r}, \ldots, \bar{r})$ denotes a constant trajectory.
\end{problem}
\end{tcolorbox}

\section{Reference Condensation}
\label{sec:refcond}

\subsection{Optimal Reference Condensation}

We first show that the optimal control law in~\eqref{eq:tracking-lqc} is affine in the state and reference. This leads directly to a closed-form solution of Problem~\ref{prob:ref-cond}.

\begin{lemma}[Tracking LQR]
    \label{lem:track-lqc}
    The optimal policy defined in~\eqref{eq:tracking-lqc} is affine:
    \begin{equation}
        \label{eq:u-closed}
        \mathbf{u}^*(x_0, \mathbf{r}) = F_x x_0 + F_r \mathbf{r},
    \end{equation}
    where $F_x \in \mathbb{R}^{Nn_u \times n_x}$ and $F_r \in \mathbb{R}^{Nn_u \times Nn_r}$ depend on $A$, $B$, $C$, $Q$, $R$, and~$N$.
\end{lemma}
\begin{IEEEproof}
    See the Appendix for explicit expressions of $F_x$ and $F_r$.
\end{IEEEproof}

This affine structure immediately yields a closed-form solution to Problem~\ref{prob:ref-cond}:

\begin{tcolorbox}
\begin{theorem}[Reference condensation: closed form]
    \label{th:closedform-refcond}
    A minimum-norm solution to Problem~\ref{prob:ref-cond} is $\bar{r} = S \mathbf{r}$, where
    \begin{equation}
        \label{eq:S-matrix}
        S \triangleq (F_r \mathbf{I})^\dagger F_r \in \mathbb{R}^{n_r \times Nn_r}.
    \end{equation}
    If $F_r \mathbf{I}$ has full column rank, then this solution is unique. In all cases, every minimizer is independent of the state~$x_0$.
\end{theorem}
\end{tcolorbox}
\begin{IEEEproof}
    Substituting~\eqref{eq:u-closed} into Problem~\ref{prob:ref-cond} gives
    \begin{align}
        \bar{r} &= \argmin_{\bar{r}} \left\|F_x x_0 + F_r \mathbf{r} - F_x x_0 - F_r \mathbf{I} \bar{r}\right\|_2^2 \notag \\
                &= \argmin_{\bar{r}} \left\|F_r \mathbf{r} - F_r \mathbf{I} \bar{r}\right\|_2^2, \label{eq:ls-condensation}
    \end{align}
    which is a least-squares problem. The pseudoinverse expression $\bar{r} = (F_r \mathbf{I})^\dagger F_r \mathbf{r}$ gives its minimum-norm solution, and this solution is unique when $F_r \mathbf{I}$ has full column rank. Since the $x_0$-dependent terms cancel, every minimizer depends only on~$\mathbf{r}$.
\end{IEEEproof}

\begin{remark}[Geometric interpretation]
\label{rem:geometry}
The condensation has a simple geometric interpretation. The preview control sequence~$F_r \mathbf{r}$ lies in~$\mathbb{R}^{Nn_u}$. The control sequences achievable with a constant reference form the subspace $\col(F_r \mathbf{I}) \subseteq \mathbb{R}^{Nn_u}$. The condensation chooses $\bar{r}$ so that $F_r \mathbf{I} \bar{r}$ is the orthogonal projection of~$F_r \mathbf{r}$ onto~$\col(F_r \mathbf{I})$.
\end{remark}

\subsection{Properties of the Condensation}

We now state several useful properties of the condensation.

\begin{lemma}[Affine combination]
    \label{lem:affine}
    If $F_r \mathbf{I}$ has full column rank, then all row sums of~$S$ equal one. That is, $\bar{r}$ is an affine combination of the references~$r_1, \ldots, r_N$.
\end{lemma}
\begin{IEEEproof}
    The row sums of~$S$ are $S \mathbf{1}_{Nn_r}$. Recall that $\mathbf{I} = \mathbf{I}_{n_r \times N}$ denotes $N$ copies of~$I_{n_r}$ stacked vertically, so $\mathbf{1}_{Nn_r} = \mathbf{I} \cdot \mathbf{1}_{n_r}$. Using~\eqref{eq:S-matrix} then gives
    \begin{equation*}
        S \mathbf{1}_{Nn_r} = (F_r \mathbf{I})^\dagger F_r \mathbf{I} \cdot \mathbf{1}_{n_r} = I_{n_r} \cdot \mathbf{1}_{n_r} = \mathbf{1}_{n_r},
    \end{equation*}
    where the second equality uses that the pseudoinverse is a left inverse for full-column-rank matrices~\cite{golub2013matrix}.
\end{IEEEproof}

\begin{remark}
    The full-column-rank condition in Lemma~\ref{lem:affine} is a nondegeneracy assumption on the map from a constant reference to the induced open-loop control sequence. It depends on $(A,B,C,Q,N)$ and holds in our numerical examples.
\end{remark}

\begin{corollary}[Exact recovery]
    \label{cor:exact-recovery}
    If $r_k = c$ for all $k = 1, \ldots, N$ for some $c \in \mathbb{R}^{n_r}$, then $\bar{r} = c$. That is, constant references are preserved exactly by the condensation.
\end{corollary}
\begin{IEEEproof}
    If $\mathbf{r} = \mathbf{I} c$, then $\bar{r} = S \mathbf{I} c = (F_r \mathbf{I})^\dagger F_r \mathbf{I} \, c = c$.
\end{IEEEproof}

\begin{proposition}[Error bound]
    \label{prop:error-bound}
    The approximation error in the optimal control sequence satisfies
    \begin{equation}
        \label{eq:error-bound}
        \left\|\mathbf{u}^*(x_0, \mathbf{r}) - \mathbf{u}^*(x_0, \mathbf{I}\bar{r})\right\|_2 \leq \bar{\sigma}(F_r)\, \left\|\mathbf{r} - \mathbf{I} \bar{r}_{\mathrm{avg}}\right\|_2,
    \end{equation}
    where $\bar{r}_{\mathrm{avg}} = \frac{1}{N}\sum_{k=1}^N r_k$ is the average reference and $\bar{\sigma}(F_r)$ is the largest singular value of~$F_r$.
\end{proposition}
\begin{IEEEproof}
    From~\eqref{eq:ls-condensation} we get that 
    \begin{equation}
        \begin{aligned}
            \left\|\mathbf{u}^*(x_0, \mathbf{r}) - \mathbf{u}^*(x_0, \mathbf{I}\bar{r})\right\|_2 &= \|F_r(\mathbf{r} - \mathbf{I}\bar{r})\|_2  \\
                                                                                                  &\leq \|F_r(\mathbf{r} - \mathbf{I}\bar{r}_{\mathrm{avg}})\|_2 \\
                                                                                                  &\leq \bar{\sigma}(F_r)\,\|\mathbf{r} - \mathbf{I}\bar{r}_{\mathrm{avg}}\|_2,
        \end{aligned}
    \end{equation}   
    where the first inequality uses optimality of~$\bar{r}$ (any constant reference gives an upper bound) and the second uses the submultiplicativity of the spectral norm.
\end{IEEEproof}

\begin{remark}
    The term $\|\mathbf{r} - \mathbf{I}\bar{r}_{\mathrm{avg}}\|_2^2 = \sum_{k=1}^N \|r_k - \bar{r}_{\mathrm{avg}}\|_2^2$ measures how much the reference varies around its mean over the horizon. Slowly varying references therefore lead to small approximation error.
\end{remark}

\begin{proposition}[Closed-loop tracking error bound]
    \label{prop:closedloop-bound}
    Consider the receding-horizon implementation of the unconstrained controller \eqref{eq:tracking-lqc}. Let $\{x_k^{\mathrm{prev}}\}$ and $\{x_k^{\mathrm{cond}}\}$ denote the closed-loop state trajectories under the preview and condensation controllers, respectively, starting from the same initial state $x_0^{\mathrm{prev}} = x_0^{\mathrm{cond}}$.
    Define $\Delta x_k \triangleq x_k^{\mathrm{prev}} - x_k^{\mathrm{cond}}$ and $A_{\mathrm{cl}} \triangleq A + B [F_x]_1$, where $[F_x]_1$ denotes the first $n_u$ rows of~$F_x$.
    Let $\mathbf{r}^{(j)}$ denote the preview window available at time step~$j$, and define $e_j \triangleq \mathbf{r}^{(j)} - \mathbf{I} S \mathbf{r}^{(j)}$.
    Suppose there exist constants $c \geq 1$ and $\lambda \in [0,1)$ such that $\|A_{\mathrm{cl}}^i\|_2 \leq c \lambda^i$ for all $i \geq 0$.
    Then
    \begin{equation}
        \label{eq:closedloop-bound}
        \|\Delta x_k\|_2 \leq \frac{c\,\|B [F_r]_1\|_2}{1 - \lambda} \max_{j \geq 0} \|e_j\|_2,
    \end{equation}
    where $[F_r]_1$ denotes the first $n_u$ rows of~$F_r$.
\end{proposition}
\begin{IEEEproof}
    Both controllers share the same plant. At time step~$k$, the preview controller applies $u_k^{\mathrm{prev}} = [F_x]_1 x_k^{\mathrm{prev}} + [F_r]_1 \mathbf{r}^{(k)}$ and the condensation controller applies $u_k^{\mathrm{cond}} = [F_x]_1 x_k^{\mathrm{cond}} + [F_r \mathbf{I}]_1 S \mathbf{r}^{(k)}$. Subtracting the state updates yields
    \begin{equation*}
        \Delta x_{k+1} = A_{\mathrm{cl}} \, \Delta x_k  + B [F_r]_1 e_k.
    \end{equation*}
    Since $\Delta x_0 = 0$, repeated substitution gives
    \begin{equation*}
        \Delta x_k = \sum_{i=0}^{k-1} A_{\mathrm{cl}}^{k-1-i} B [F_r]_1 e_i.
    \end{equation*}
    Taking norms and using $\|A_{\mathrm{cl}}^i\|_2 \leq c \lambda^i$ gives
    \begin{equation*}
        \begin{aligned}
            \|\Delta x_k\|_2 &\leq c\, \|B [F_r]_1\|_2 \sum_{i=0}^{k-1} \lambda^{k-1-i} \|e_i\|_2 \\
                             &\leq c\, \|B [F_r]_1\|_2 \sum_{i=0}^{k-1} \lambda^i \max_{j \geq 0}\|e_j\|_2.
        \end{aligned}
    \end{equation*}
    Finally,  $\sum_{i=0}^{k-1}\lambda^i \leq 1/(1-\lambda)$ gives~\eqref{eq:closedloop-bound}.
\end{IEEEproof}

\begin{remark}
    Proposition~\ref{prop:closedloop-bound} is an input-to-state bound. The residual $e_j = \mathbf{r}^{(j)} - \mathbf{I} S \mathbf{r}^{(j)}$ is the part of the preview window that cannot be represented by one setpoint, and it acts as a disturbance on the closed-loop error dynamics. If each preview window is well approximated by its condensed setpoint, then the preview and condensation trajectories stay close. For constant references, $e_j = 0$ for all~$j$, so $\Delta x_k = 0$ for all~$k$ by Corollary~\ref{cor:exact-recovery}. The exponential-stability assumption on $A_{\mathrm{cl}}$ is standard for stabilizable/detectable LQR settings.
\end{remark}

\subsection{Weighted Condensation}
\label{sec:weighted}
In MPC, only the first control action~$u_0^*$ is applied before the problem is solved again. It is therefore most important to preserve~$u_0^*$ accurately. This motivates a weighted variant of Problem~\ref{prob:ref-cond}:
\begin{equation}
    \label{eq:weighted-condensation}
    \bar{r}_W = \argmin_{\bar{r}} \left\|F_r \mathbf{r} - F_r \mathbf{I} \bar{r}\right\|_W^2,
\end{equation}
where $W \succ 0$ is a weighting matrix. If $F_r \mathbf{I}$ has full column rank, then the unique minimizer is
\begin{equation}
    \label{eq:S-weighted}
    \bar{r}_W = S_W \mathbf{r}, \quad S_W \triangleq (\mathbf{I}^\top F_r^\top W F_r \mathbf{I})^{-1} \mathbf{I}^\top F_r^\top W F_r.
\end{equation}
Choosing $W = \mathrm{diag}(\rho I_{n_u}, I_{n_u}, \ldots, I_{n_u})$ with $\rho \gg 1$ gives more weight to~$u_0^*$, at the expense of lower accuracy for later control actions.

\begin{remark}[Weighted projection and practical use]
    The control sequence $F_r \mathbf{I}\bar{r}_W$ is the orthogonal projection of $F_r\mathbf{r}$ onto $\col(F_r \mathbf{I})$ in the inner product induced by~$W$. The unweighted condensation is recovered with $W = I$. In receding-horizon MPC only $u_0^*$ is applied before the problem is solved again, so diagonal weights with a larger first block are natural. Our constrained-MPC experiments use this weighted condensation unless stated otherwise.
\end{remark}

\section{Application to Model Predictive Control}
\label{sec:mpc}

Reference condensation is especially useful for MPC, where a problem of the form~\eqref{eq:tracking-lqc}, possibly with constraints, is solved repeatedly in receding-horizon form. The exact closed-form optimality result applies to the unconstrained tracking problem. In constrained MPC, condensation is used as a low-dimensional preprocessing step. In practice we use the weighted condensation map~$S_W$ because only the first control is applied online.

\subsection{Constrained MPC with Reference Condensation}

Consider a constrained MPC formulation
\begin{equation}
    \label{eq:mpc}
    \begin{aligned}
        u_0^*(x, \mathbf{r}) = \argmin_{u_0, \ldots, u_{N-1}} \;& J(\mathbf{u}; x, \mathbf{r}) \\
        \text{s.t.} \;& x_{k+1} = A x_k + B u_k, \\
        & u_k \in \mathcal{U}, \; x_k \in \mathcal{X}, \\
        & x_0 = x,
    \end{aligned}
\end{equation}
where $\mathcal{U}$ and $\mathcal{X}$ are polyhedral constraint sets. At each sampling instant, the controller receives the measured state~$x$ and a preview trajectory~$\mathbf{r} = (r_1, \ldots, r_N)$, but only the first control action $u_0^*$ is applied.

With reference condensation, the preview trajectory is first compressed to a constant setpoint $\bar{r}_c = S_W\mathbf{r}$. The MPC problem is then solved with the constant reference $\mathbf{I}\bar{r}_c$, yielding $u_0^*(x, \mathbf{I}\bar{r}_c)$. Algorithm~\ref{alg:refcond} summarizes the procedure.

\begin{algorithm}[t]
    \caption{MPC with Reference Condensation}
    \label{alg:refcond}
    \begin{algorithmic}[1]
        \Require System $(A,B,C)$, weights $(Q,R)$, horizon~$N$
        \State \textbf{Offline:} Compute $F_r$ via Lemma~\ref{lem:track-lqc} and $S_W$ via \eqref{eq:S-weighted}. 
        \State \textbf{Online:} At each time step, given state~$x$ and preview~$\mathbf{r}$:
        \State \quad $\bar{r}_c \gets S_W \mathbf{r}$
        \State \quad Solve MPC~\eqref{eq:mpc} with constant reference $\mathbf{I}\bar{r}_c$
        \State \quad Apply $u_0^*$ to the system
    \end{algorithmic}
\end{algorithm}

\begin{remark}[Optimality]
    For the unconstrained problem~\eqref{eq:tracking-lqc}, the map $\bar{r} = S\mathbf{r}$ is the optimal constant reference for the unweighted least-squares criterion, and $\bar{r}_W = S_W\mathbf{r}$ is optimal for the weighted criterion~\eqref{eq:weighted-condensation}. For constrained MPC~\eqref{eq:mpc}, both are generally suboptimal because the optimal control law is piecewise affine rather than affine. We therefore use condensation in constrained MPC as a structured approximation, and evaluate its closed-loop quality empirically in Section~\ref{sec:experiments}.
\end{remark}

\begin{remark}[Feasible-set invariance]
    Since the reference enters only through the objective in~\eqref{eq:mpc}, the feasible set is independent of the reference. Hence, \eqref{eq:mpc} has the same feasibility properties under a condensed reference as under full preview.
\end{remark}

\subsection{Complexity Reduction for Explicit MPC}
\label{sec:complexity}

For \emph{explicit} MPC~\cite{bemporad2002explicit}, where the control law is pre-computed as a piecewise affine function of the parameters and stored as a lookup table, reference condensation can reduce complexity dramatically.

Without condensation, preview MPC has $n_x + Nn_r$ parameters: the state and the full reference trajectory. With condensation, this drops to $n_x + n_r$: the state and the condensed setpoint~$\bar{r}$. This matters in three ways. First, the cost of solving the parametric quadratic program and enumerating critical regions grows quickly with parameter dimension~\cite{tondel2003evaluation}, so full-preview explicit MPC soon becomes difficult to handle. Second, each critical region stores a polyhedral description in $\mathbb{R}^{n_\theta}$ and an affine control law, so storage per region scales linearly with~$n_\theta$. Third, online point location is performed in $\mathbb{R}^{n_\theta}$, so its cost also depends on~$n_\theta$. By keeping $n_\theta = n_x + n_r$ for any horizon, reference condensation makes preview-capable explicit MPC much more practical.

\begin{remark}[Online overhead]
    Computing $\bar{r} = S\mathbf{r}$ requires $n_r \cdot Nn_r$ multiplications at each time step. This is negligible compared to solving a QP online and comparable to the cost of one lookup in an explicit MPC table. In practice, reference condensation adds almost no additional online cost.
\end{remark}

\begin{remark}[Applicability beyond MPC]
    Although we focus on MPC, reference condensation applies to any setpoint-tracking controller. If a controller has the form $u = \kappa(x, r)$ for a constant setpoint~$r$, then condensation gives a preview-aware version $u = \kappa(x, S\mathbf{r})$. This includes, for example, predictive PID controllers.
\end{remark}

\section{Numerical Experiments}
\label{sec:experiments}

First, we consider the continuous-time double integrator $\ddot{p} = u$ (position~$p$, velocity~$v = \dot{p}$), discretized with zero-order hold at sampling time $T_s = 0.1\,$s, resulting in the state-space system
\begin{equation}
    \label{eq:double-int}
    x_{k+1} = \begin{pmatrix} 1 & 0.1 \\ 0 & 1 \end{pmatrix} x_k + \begin{pmatrix} 0.005 \\ 0.1 \end{pmatrix} u_k,
\end{equation}
where $x = (p, v)^\top$. The tracking output is position: $C = \begin{pmatrix} 1 & 0 \end{pmatrix}$, with weights $Q = 1$ and $R = 1$ for simplicity. The prediction horizon is $N = 50$ (covering $5\,$s), and the control is constrained to $|u| \leq 1$. The reference is a unit step at $t = 5\,$s:
\begin{equation}
    \label{eq:ref-step}
    r(t) = \begin{cases} 1 & \text{if } t \geq 5, \\ 0 & \text{otherwise.} \end{cases}
\end{equation}

Four MPC controllers are compared:
\begin{itemize}
    \item \textbf{No preview}: uses only the current reference as setpoint.
    \item \textbf{Average ref.}: uses the arithmetic mean of the future references as setpoint, i.e., $\bar{r} = \frac{1}{N}\sum_{k=1}^N r_k$.
    \item \textbf{Ref.\ condensation} (proposed): compresses the future trajectory into one setpoint as described in Section~\ref{sec:refcond}.
    \item \textbf{Preview}: uses the full future reference trajectory.
\end{itemize}
All controllers solve the same constrained QP~\eqref{eq:mpc} (using the solver \texttt{DAQP} \cite{arnstrom2022dual}) at each step. Only the reference information changes. We use the weighted condensation in the main closed-loop comparisons, motivated by the results in Table~\ref{tab:weighted}. We have implemented reference condensation in the Julia package \texttt{LinearMPC.jl}, and code for reproducing all results in this paper is available at \url{https://github.com/darnstrom/refcond-experiments}.

\subsection{Tracking Performance}

\begin{figure}[t]
    \centering
    \begin{tikzpicture}[scale=1]
        \begin{axis}[
            xmin = 0, xmax=15,
            scale=0.95,
            xlabel={Time [s]},
            ylabel={Position $p$},
            legend style={at={(1.00,0.0)},anchor=south east},
            ymajorgrids,xmajorgrids,
            x post scale=1,
            y post scale=0.5,
            legend cell align={left},legend columns=1,
            legend style={nodes={scale=0.8, transform shape}},
            ]
            \addplot [black,very thick,dashed] table [x={t}, y={r1}] {\stepnopreview};
            \addplot [set19c1,thick] table [x={t}, y={x1}] {\stepnopreview};
            \addplot [set19c5,thick,dashdotted] table [x={t}, y={x1}] {\stepavgref};
            \addplot [set19c2,thick] table [x={t}, y={x1}] {\steppreview};
            \addplot [set19c3,thick,densely dotted] table [x={t}, y={x1}] {\steprefcond};
            \legend{Reference, No preview, Average ref., Preview, Ref.\ condensation}
        \end{axis}
    \end{tikzpicture}
    \caption{Step responses of four MPC controllers. The preview and proposed (reference condensation) controllers anticipate the step change at $t=5\,$s and begin responding earlier, achieving faster tracking. The average-reference baseline provides partial anticipation but substantially worse tracking than the proposed method.}
    \label{fig:preview}
\end{figure}
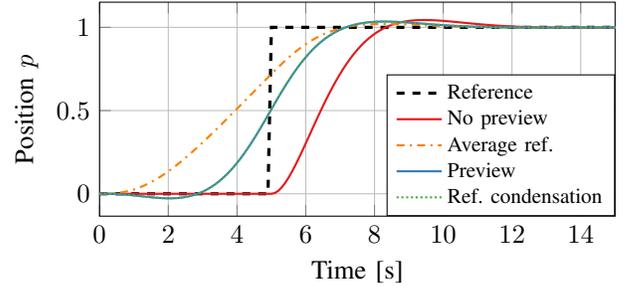

Fig.~\ref{fig:preview} shows the position trajectories. Both the preview and reference-condensation controllers anticipate the step and begin accelerating before $t = 5\,$s, which gives much faster tracking than no preview. The average-reference baseline also anticipates the change, but tracks it less well. The reference-condensation trajectory is visually indistinguishable from full preview.

\begin{figure}[t]
    \centering
    \begin{tikzpicture}[scale=1]
        \begin{axis}[
            xmin = 1, xmax=50,
            ymin = -0.02, ymax = 0.06,
            scale=0.95,
            xlabel={Horizon step $k$},
            ylabel={Weight $S_{1k}$},
            ymajorgrids,xmajorgrids,
            x post scale=1,
            y post scale=0.35,
            ]
            \addplot [set19c2,thick,mark=*,mark size=1.2pt] table [x={k}, y={weight}] {\condensationmatrix};
            \addplot [black,thin,dashed,domain=0:52] {0};
        \end{axis}
    \end{tikzpicture}
    \caption{Entries of the condensation matrix~$S$ (row vector for SISO). Near-term references receive the highest weight, with an approximately exponential decay. The weights sum to one (affine combination, Lemma~\ref{lem:affine}).}
    \label{fig:weights}
\end{figure}
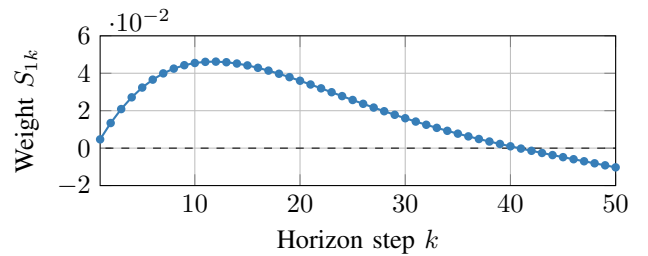

Fig.~\ref{fig:weights} shows the entries of the $1 \times 50$ condensation matrix~$S$ for this example. Near-term references ($k \in \{1, \ldots, 15\}$) receive the largest positive weights, with a peak around $k \approx 10$, while distant references receive small or slightly negative weights. The weights sum to one, consistent with Lemma~\ref{lem:affine}. This reflects the simple fact that near-future references affect the current control action more than distant ones.

\begin{figure}[t]
    \centering
    \begin{tikzpicture}[scale=1]
        \begin{axis}[
            xmin = 0, xmax=15,
            scale=0.95,
            xlabel={Time [s]},
            ylabel={Position $p$},
            legend style={at={(0.5,1.05)},anchor=south},
            ymajorgrids,xmajorgrids,
            x post scale=1,
            y post scale=0.5,
            legend cell align={left},legend columns=3,
            legend style={nodes={scale=0.8, transform shape}},
            ]
            \addplot [black,very thick,dashed] table [x={t}, y={r1}] {\sinnopreview};
            \addplot [set19c1,thick] table [x={t}, y={x1}] {\sinnopreview};
            \addplot [set19c5,thick,dashdotted] table [x={t}, y={x1}] {\sinavgref};
            \addplot [set19c2,thick] table [x={t}, y={x1}] {\sinpreview};
            \addplot [set19c3,thick,densely dotted] table [x={t}, y={x1}] {\sinrefcond};
            \legend{Reference, No preview, Average ref., Preview, Ref.\ condensation}
        \end{axis}
    \end{tikzpicture}
    \caption{Sinusoidal reference tracking ($r(t) = \sin(0.5t)$, $N=50$). Preview and reference condensation achieve nearly identical tracking, both dramatically outperforming no preview and the average-reference baseline. ISE values are reported in Table~\ref{tab:ise}.}
    \label{fig:sinusoidal}
\end{figure}
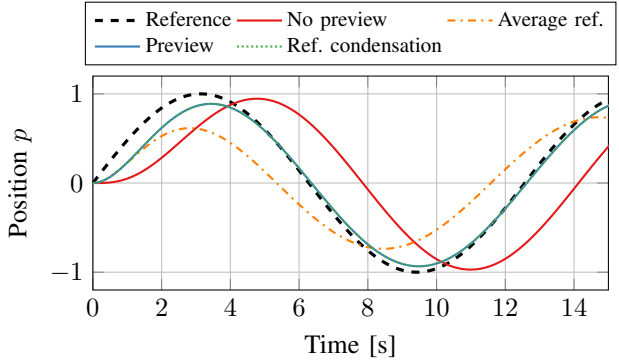

To quantify tracking, we compute the integrated squared error (ISE) $\int_0^T (p(t) - r(t))^2 \, dt$. Table~\ref{tab:ise} reports the ISE for the step and sinusoidal references. For the step reference, the ISE drops from $1.114$ with no preview to $0.628$ with average reference and to $0.259$ with preview and weighted reference condensation, corresponding to reductions of $44\%$ and $77\%$. The condensation causes no measurable loss relative to full preview, even though it reduces the reference information from~$50$ parameters to one scalar. The average-reference baseline captures only part of the preview benefit, which shows the value of using the control-sensitive condensation map~$S_W$ rather than a plain arithmetic average.

\begin{table}[t]
    \centering
    \caption{Integrated squared error (ISE) for step and sinusoidal references with horizon $N=50$.}
    \label{tab:ise}
    \begin{tabular}{l c c c c}
        \hline
        Reference & No preview & Average ref. & Ref.\ cond. & Preview \\
        \hline
        Step  & $1.114$ & $0.628$ & $\mathbf{0.259}$ & $0.259$ \\
        Sinusoid & $5.205$ & $2.26$ & $\mathbf{0.19}$ & $0.19$ \\
        \hline
    \end{tabular}
\end{table}

Finally, Fig.~\ref{fig:sinusoidal} shows tracking of the sinusoid $r(t) = \sin(0.5t)$ with horizon $N = 50$. The no-preview controller has clear phase lag, while the preview and weighted-condensation controllers track the reference closely. As reported in Table~\ref{tab:ise}, the average-reference baseline improves over no preview (ISE: $2.26$ vs.\ $5.205$), but it is still clearly worse than reference condensation (ISE: $0.19$). Again, reference condensation matches full preview.

\subsection{Horizon Study}

Table~\ref{tab:horizon} reports the ISE for different prediction horizons. The weighted condensation matches the preview ISE at every horizon. The preview benefit saturates around $N = 20$: beyond that, longer horizons do not help much because the full step transition is already visible.

The average-reference baseline performs well at moderate horizons ($N = 20$, ISE $0.29$) but \emph{degrades} at long horizons ($N = 100$, ISE $1.56$, worse than no preview). The reason is simple: the arithmetic mean averages over the whole horizon, so at long horizons the step is diluted by the many constant-reference samples that follow it. Reference condensation avoids this by weighting the references according to their effect on the control sequence, with the main experiments further emphasizing the first applied control.

\begin{table}[t]
    \centering
    \caption{Integrated squared error (ISE) for the step response as a function of the prediction horizon~$N$.}
    \label{tab:horizon}
    \begin{tabular}{r c c c c}
        \hline
        $N$ & No preview & Average ref. & Ref.\ cond. & Preview \\
        \hline
         5  & 5.47  & 5.20 & 5.14 & 5.14 \\
        10  & 1.99  & 1.42 & 1.30 & 1.30 \\
        20  & 1.07  & 0.29 & 0.23 & 0.23 \\
        50  & 1.11  & 0.63 & 0.26 & 0.26 \\
        75  & 1.11  & 1.18 & 0.26 & 0.26 \\
        100 & 1.11  & 1.56 & 0.26 & 0.26 \\
        \hline
    \end{tabular}
\end{table}

For short horizons ($N = 5$), the preview window does not reach the step transition when $t < 4.5\,$s, so the benefit is small. For $N \geq 20$, the controller can see the step up to $2\,$s ahead, which gives the anticipatory behavior in Fig.~\ref{fig:preview}. The condensation captures this benefit for every horizon.

\subsection{Aircraft Benchmark}

Next, we consider the aircraft example from \cite{aircraft1988}, which is a fourth-order, two-input, two-output system. We use horizon $N = 20$ and apply a square-wave reference of amplitude $0.05$ to the second regulated output, with switches at $t = 12\,$s and $t = 20\,$s. Table~\ref{tab:aircraft} reports the ISE summed over both regulated outputs.

\begin{table}[t]
    \centering
    \caption{Integrated squared error (ISE) for the aircraft benchmark with a square wave as reference.} 
    \label{tab:aircraft}
    \begin{tabular}{c c c c}
        \hline
        No preview & Average ref. & Ref.\ cond. & Preview \\
        \hline
        $0.01296$ & $0.00496$ & $\mathbf{0.00469}$ & $0.00469$ \\
        \hline
    \end{tabular}
\end{table}

The same pattern appears as in the double-integrator example. Preview reduces the ISE by about $63\%$ relative to no preview, the average-reference baseline captures part of the gain, and reference condensation is again indistinguishable from full preview. This shows that the result is not limited to a scalar second-order system.

\subsection{Weighted Condensation}
\label{sec:weighted-eval}

Next, we study the weighted condensation further. We generated $40$ random piecewise-constant references of length $15\,$s for the same constrained double integrator, using amplitudes clipped to~$[-1,1]$, random dwell times between $1$ and $3\,$s, horizon~$N = 30$. Table~\ref{tab:weighted} reports the mean closed-loop ISE normalized by the preview ISE and the mean first-control mismatch $|u_0^{\mathrm{cond}} - u_0^{\mathrm{prev}}|$.

\begin{table}[t]
    \centering
    \caption{Random-reference study over 40 piecewise-constant trajectories. Moderate or larger first-control weighting nearly recovers full preview on both closed-loop tracking and the first applied control action.}
    \label{tab:weighted}
    \begin{tabular}{l c c}
        \hline
        Weighting & Mean ISE / Preview & Mean $|u_0 - u_0^{\mathrm{prev}}|$ \\
        \hline
        Uniform ($\rho = 1$) & $1.0258$ & $4.81 \times 10^{-2}$ \\
        Moderate ($\rho = 10^2$) & $1.0000$ & $1.66 \times 10^{-5}$ \\
        Strong ($\rho = 10^6$) & $1.0000$ & $1.66 \times 10^{-13}$ \\
        \hline
    \end{tabular}
\end{table}

This study makes the role of~$S_W$ clearer than the step example alone. Uniform weighting is already close in ISE, but it still perturbs the first control. Moderate or larger emphasis on~$u_0^*$ almost eliminates that mismatch and, in these trials, makes the closed-loop ISE indistinguishable from full preview. For the strongest weight, the remaining first-control mismatch is at machine precision. This is why the rest of the constrained-MPC study uses the strongly weighted condensation.

\subsection{Complexity Reduction for Explicit MPC}

To illustrate the complexity reduction in explicit MPC, Table~\ref{tab:complexity} reports the parameter dimension~$n_\theta$ and the number of critical regions for the double-integrator example at several horizons. The explicit solutions are computed with the parametric QP solver of~\cite{arnstrom2024pdaqp}. The parameter dimension and number of critical regions directly relate to the memory footprint, which is often the bottleneck for using explicit MPC. Moreover, they also directly correlate with the tractability of constructing a binary search tree for efficient evaluation of the resulting piecewise-affine control law online \cite{tondel2003evaluation}.

\begin{table}[t]
    \centering
    \caption{Explicit MPC complexity for the double-integrator example. Ref.\ condensation and no-preview MPC have $n_\theta = 3$ parameters for all horizons. Full-preview MPC has $n_\theta = 2 + N$, its region count grows rapidly with $N$, and for $N=50$ we could not compute the explicit solution in our setup.}
    \label{tab:complexity}
    \begin{tabular}{r | c c | c c}
        \hline
        & \multicolumn{2}{c}{No preview / Ref.\ cond.} & \multicolumn{2}{c}{Full preview} \\
        $N$ & $n_\theta$ & Regions & $n_\theta$ & Regions \\
        \hline
         5  & 3 &     3 &  7 &     3 \\
        10  & 3 &    13 & 12 &    13 \\
        15  & 3 &    23 & 17 &    53 \\
        20  & 3 &    43 & 22 & 909  \\
        25  & 3 &    107 & 27 & 20449 \\
        50  & 3 & 1{,}467 & 52 & \multicolumn{1}{c}{---} \\
        \hline
    \end{tabular}
\end{table}

For the short horizons $N \in \{5,10\}$, the two formulations produce the same number of critical regions. In this range, the extra preview parameters change the linear term of the parametric QP but do not yet change which constraints become active. At $N=15$, the picture changes: full preview jumps to $53$ regions, while no preview and reference condensation stay at $23$. From there, the larger preview parameter vector keeps increasing the critical regions, reaching $909$ regions at $N=20$ and $20{,}449$ at $N=25$, while reference condensation still keeps $n_\theta = 3$.
For $N = 50$, the gap is larger still: $n_\theta = 52$ for full preview versus $n_\theta = 3$ for condensation, and the condensed explicit solution has $1{,}467$ regions. We could not compute the full-preview solution at $N=50$ in our setup. Moreover, each full-preview region would require a polyhedral description in $\mathbb{R}^{52}$ and an affine control law with $53$ coefficients, compared to $\mathbb{R}^3$ and $4$ coefficients with condensation.
The bottom line is: \textit{reference condensation makes preview-capable explicit MPC much more practical.}

\section{Conclusion}
\label{sec:conclusion}

We introduced reference condensation, a method that compresses a future reference trajectory into one setpoint through a linear map. Numerical experiments on a constrained double integrator and a higher-order aircraft benchmark showed that the weighted condensation tracks almost identically to full preview while reducing the parameter dimension from $n_x + Nn_r$ to $n_x + n_r$. For explicit MPC, the reduction is dramatic: at horizon $N=25$, reference condensation yields 107 critical regions in dimension 3, while full preview requires $20{,}449$ critical regions in dimension 27.

\bibliographystyle{IEEEtran}
{\footnotesize\bibliography{lib}}

\begin{thebibliography}{10}
\providecommand{\url}[1]{#1}
\csname url@samestyle\endcsname
\providecommand{\newblock}{\relax}
\providecommand{\bibinfo}[2]{#2}
\providecommand{\BIBentrySTDinterwordspacing}{\spaceskip=0pt\relax}
\providecommand{\BIBentryALTinterwordstretchfactor}{4}
\providecommand{\BIBentryALTinterwordspacing}{\spaceskip=\fontdimen2\font plus
\BIBentryALTinterwordstretchfactor\fontdimen3\font minus \fontdimen4\font\relax}
\providecommand{\BIBforeignlanguage}[2]{{%
\expandafter\ifx\csname l@#1\endcsname\relax
\typeout{** WARNING: IEEEtran.bst: No hyphenation pattern has been}%
\typeout{** loaded for the language `#1'. Using the pattern for}%
\typeout{** the default language instead.}%
\else
\language=\csname l@#1\endcsname
\fi
#2}}
\providecommand{\BIBdecl}{\relax}
\BIBdecl

\bibitem{birla2015optimal}
N.~Birla and A.~Swarup, ``Optimal preview control: A review,'' \emph{Optimal Control Applications and Methods}, vol.~36, no.~2, pp. 241--268, 2015.

\bibitem{tomizuka1975optimal}
M.~Tomizuka, ``Optimal continuous finite preview problem,'' \emph{IEEE Transactions on Automatic Control}, vol.~20, no.~3, pp. 362--365, 1975.

\bibitem{rawlings2017model}
J.~B. Rawlings, D.~Q. Mayne, and M.~Diehl, \emph{Model Predictive Control: Theory, Computation, and Design}, 2nd~ed.\hskip 1em plus 0.5em minus 0.4em\relax Nob Hill Publishing, 2017.

\bibitem{borrelli2017predictive}
F.~Borrelli, A.~Bemporad, and M.~Morari, \emph{Predictive Control for Linear and Hybrid Systems}.\hskip 1em plus 0.5em minus 0.4em\relax Cambridge University Press, 2017.

\bibitem{bemporad2002explicit}
A.~Bemporad, M.~Morari, V.~Dua, and E.~N. Pistikopoulos, ``The explicit linear quadratic regulator for constrained systems,'' \emph{Automatica}, vol.~38, no.~1, pp. 3--20, 2002.

\bibitem{alessio2009survey}
A.~Alessio and A.~Bemporad, ``A survey on explicit model predictive control,'' in \emph{Nonlinear Model Predictive Control}.\hskip 1em plus 0.5em minus 0.4em\relax Springer, 2009, pp. 345--369.

\bibitem{garone2017reference}
E.~Garone, S.~Di~Cairano, and I.~Kolmanovsky, ``Reference and command governors for systems with constraints: A survey on theory and applications,'' \emph{Automatica}, vol.~75, pp. 306--328, 2017.

\bibitem{limon2008mpc}
D.~Limon, I.~Alvarado, T.~Alamo, and E.~F. Camacho, ``{MPC} for tracking piecewise constant references for constrained linear systems,'' \emph{Automatica}, vol.~44, no.~9, pp. 2382--2387, 2008.

\bibitem{krupa2024model}
P.~Krupa, J.~K{\"o}hler, A.~Ferramosca, I.~Alvarado, M.~N. Zeilinger, T.~Alamo, and D.~Limon, ``Model predictive control for tracking using artificial references: Fundamentals, recent results and practical implementation,'' in \emph{2024 IEEE 63rd Conference on Decision and Control (CDC)}.\hskip 1em plus 0.5em minus 0.4em\relax IEEE, 2024, pp. 2977--2991.

\bibitem{anderson2007optimal}
B.~D.~O. Anderson and J.~B. Moore, \emph{Optimal Control: Linear Quadratic Methods}.\hskip 1em plus 0.5em minus 0.4em\relax Courier Corporation, 2007.

\bibitem{golub2013matrix}
G.~H. Golub and C.~F. Van~Loan, \emph{Matrix Computations}, 4th~ed.\hskip 1em plus 0.5em minus 0.4em\relax Johns Hopkins University Press, 2013.

\bibitem{tondel2003evaluation}
P.~T{\o}ndel, T.~A. Johansen, and A.~Bemporad, ``Evaluation of piecewise affine control via binary search tree,'' \emph{Automatica}, vol.~39, no.~5, pp. 945--950, 2003.

\bibitem{arnstrom2022dual}
D.~Arnström, A.~Bemporad, and D.~Axehill, ``A dual active-set solver for embedded quadratic programming using recursive {LDL}$^{T}$ updates,'' \emph{IEEE Transactions on Automatic Control}, vol.~67, no.~8, pp. 4362--4369, 2022.

\bibitem{aircraft1988}
P.~Kapasouris, M.~Athans, and G.~Stein, ``Design of feedback control systems for stable plants with saturating actuators,'' in \emph{Proceedings of the 27th IEEE Conference on Decision and Control}, 1988, pp. 469--479 vol.1.

\bibitem{arnstrom2024pdaqp}
D.~Arnstr{\"o}m and D.~Axehill, ``A high-performant multi-parametric quadratic programming solver,'' in \emph{2024 IEEE 63rd Conference on Decision and Control (CDC)}.\hskip 1em plus 0.5em minus 0.4em\relax IEEE, 2024, pp. 303--308.

\end{thebibliography}

\appendices

\section{Proof of Lemma~\ref{lem:track-lqc}}
We write the tracking problem~\eqref{eq:tracking-lqc} in batch form. Rolling out the dynamics~\eqref{eq:dynamics} over the horizon gives
\begingroup
\small
\begin{equation*}
        \underbrace{
            \begin{pmatrix}
                x_1 \\ x_2 \\ \vdots \\ x_N
            \end{pmatrix}
        }_{\triangleq \hat{\mathbf{x}}}
        =
        \underbrace{
            \begin{pmatrix}
                A \\ A^2 \\ \vdots \\ A^{N}
            \end{pmatrix}
        }_{\triangleq \mathbf{A}}
        x_0
        +
        \underbrace{
            \begin{pmatrix}
                B & 0 & \cdots & 0 \\
                AB & B & \cdots & 0 \\
                \vdots & & \ddots &  \\
                A^{N-1}\! B& A^{N-2}\! B& \cdots & B \\
            \end{pmatrix}
        }_{\triangleq \mathbf{B}}
        \mathbf{u}.
\end{equation*}
\endgroup
Defining $\mathbf{C} = \mathrm{diag}(C, \ldots, C)$, $\mathbf{Q} = \mathrm{diag}(Q, \ldots, Q)$, and $\mathbf{R} = \mathrm{diag}(R, \ldots, R)$, the cost~\eqref{eq:cost} becomes
\begin{equation*}
    J = \|\mathbf{C}(\mathbf{A} x_0 + \mathbf{B}\mathbf{u}) - \mathbf{r}\|_{\mathbf{Q}}^2 + \|\mathbf{u}\|_{\mathbf{R}}^2.
\end{equation*}
Setting the gradient with respect to~$\mathbf{u}$ to zero yields $\mathbf{u}^* = F_x x_0 + F_r \mathbf{r}$, where
\begin{align}
    F_x &\triangleq -H^{-1} \mathbf{B}^\top \mathbf{C}^\top \mathbf{Q}\, \mathbf{C} \mathbf{A}, \label{eq:Fx} \\
    F_r &\triangleq \phantom{-}H^{-1} \mathbf{B}^\top \mathbf{C}^\top \mathbf{Q}, \label{eq:Fr}
\end{align}
and with $H \triangleq \mathbf{B}^\top \mathbf{C}^\top \mathbf{Q}\, \mathbf{C} \mathbf{B} + \mathbf{R} \succ 0$.

\end{document}